\documentclass{amsart}

\usepackage[latin1]{inputenc}
\usepackage[english]{babel}
\usepackage{indentfirst}
\usepackage{amssymb}
\usepackage{amsthm}
\usepackage{xcolor}
\usepackage[all]{xy}
\usepackage[mathscr]{eucal}

\newcommand{\impli}{\Rightarrow}
\newcommand{\Nat}{\mathbb{N}}
\newcommand{\N}{\mathbb{N}}
\newcommand{\erre}{\mathbb{R}}
\newcommand{\sub}{\subseteq}
\def\epsilon{\varepsilon}

\newtheorem{theorem}{Theorem}[section]
\newtheorem{proposition}[theorem]{Proposition}
\newtheorem{corollary}[theorem]{Corollary}
\newtheorem{lemma}[theorem]{Lemma}

\theoremstyle{definition}
\newtheorem{definition}[theorem]{Definition}
\newtheorem{example}[theorem]{Example}
\newtheorem{remark}[theorem]{Remark}

\numberwithin{equation}{section}

\title{A class of summing operators acting in spaces of operators}

\author{J. Rodr\'{i}guez}
\address{Dpto. de Ingenier\'{i}a y Tecnolog\'{i}a de Computadores\\Facultad de Inform\'{a}tica\\
Universidad de Murcia\\ 30100 Espinardo (Murcia)\\ Spain} \email{joserr@um.es}

\author{E.A. S\'{a}nchez-P\'{e}rez}
\address{Instituto Universitario de Matem\'{a}tica Pura y Aplicada\\ Universitat Polit\`{e}cnica de Val\`{e}ncia\\
Camino de Vera s/n\\ 46022 Valencia\\ Spain} \email{easancpe@mat.upv.es}



\thanks{Research partially supported by {\em Agencia Estatal de Investigaci\'{o}n} [MTM2017-86182-P to J.R. and MTM2016-77054-C2-1-P to E.A.S.P.,
both grants cofunded by ERDF, EU]; and {\em Fundaci\'on S\'eneca} [20797/PI/18 to J.R.]}

\begin{document}

\begin{abstract}
Let $X$, $Y$ and $Z$ be Banach spaces and let $U$ be a subspace of $\mathcal{L}(X^*,Y)$, the Banach
space of all operators from $X^*$ to~$Y$. An operator $S: U \to Z$
is said to be $(\ell^s_p,\ell_p)$-summing (where $1\leq p <\infty$) if there is a constant $K\geq 0$ such that
$$
	\Big( \sum_{i=1}^n \|S(T_i)\|_Z^p \Big)^{1/p}
	\le K
	\sup_{x^* \in B_{X^*}} \Big(\sum_{i=1}^n \|T_i(x^*)\|_Y^p\Big)^{1/p} 
$$
for every $n\in \Nat$ and every $T_1,\dots,T_n \in U$. 
In this paper we study this class of operators, introduced by Blasco and Signes
as a natural generalization of the $(p,Y)$-summing operators of Kislyakov. 
On one hand, we discuss Pietsch-type domination results for $(\ell^s_p,\ell_p)$-summing operators. In this direction,
we provide a negative answer to a question raised by Blasco and Signes, and we also
give new insight on a result by Botelho and Santos.
On the other hand, we extend to this setting the classical theorem of Kwapie\'{n} characterizing those
operators which factor as $S_1\circ  S_2$, where $S_2$ is absolutely $p$-summing and
$S_1^*$ is absolutely $q$-summing ($1<p,q<\infty$ and $1/p+1/q \leq 1$).  
\end{abstract}

\maketitle

\section{Introduction}

Summability of series in Banach spaces is a classical central topic in the field of mathematical analysis. 
This study is faced from an abstract point of view as a part of the general analysis of the summability properties of operators, 
using some remarkable results of the theory of operator ideals. Pietsch's Factorization Theorem is nowadays the central tool in this topic, 
and different versions of this result adapted to other contexts are currently known. This theorem establishes that operators that transform weakly $p$-summable sequences 
into absolutely $p$-summable ones can always be dominated by an integral, and factored through a subspace of an $L_p$-space.  Some related relevant 
results can also be formulated in terms of integral domination and factorization of operators. For example, recall that an operator 
between Banach spaces $S:X \to Y$ is said to be
$(p,q)$-dominated (where $1<p,q<\infty$ and $1/p+1/q=1/r\leq 1$) if for every couple of finite sequences $(x_i)_{i=1}^n$ in~$X$ and $(y_i^*)_{i=1}^n$ in~$Y^*$, the
strong $\ell_r$-norm of the sequence $(\langle S(x_i), y_i^* \rangle )_{i=1}^n$ is bounded above by the product of the
weak $\ell_p$-norm of $(x_i)_{i=1}^n$ and the weak $\ell_q$-norm of $(y_i^*)_{i=1}^n$ 
(up to a multiplying constant independent of both sequences and their length). Kwapie\'{n}'s Factorization
Theorem~\cite{kwa} states that an operator is $(p,q)$-dominated 
if and only if it can be written as the composition $S_1\circ S_2$ of operators such that $S_2$ is absolutely $p$-summing 
and the adjoint $S_1^*$ is absolutely $q$-summing (cf. \cite[\S 19]{def-flo}).  

The aim of this paper is to continue with the specific study of the summability properties of operators defined on spaces of operators. 
Throughout this paper $X$, $Y$ and $Z$ are Banach spaces.

\begin{definition}[Blasco-Signes, \cite{bla-sig}]\label{definition:pPettisSumming}
Let $1\leq p<\infty$ and let $U$ be a subspace of $\mathcal L(X^*,Y)$. An operator $S: U \to Z$ is said to be {\em $(\ell^s_p,\ell_p)$-summing} if
there is a constant $K\geq 0$ such that
\begin{equation}\label{eqn:psumming}
	\Big( \sum_{i=1}^n \|S(T_i)\|_Z^p \Big)^{1/p}
	\le K
	\sup_{x^* \in B_{X^*}} \Big(\sum_{i=1}^n \|T_i(x^*)\|_Y^p\Big)^{1/p} 
\end{equation}
for every $n\in \Nat$ and every $T_1,\dots,T_n \in U$.
\end{definition}

Some fundamental properties of this type of operators are already known, as well as the main picture of their summability properties. 
The works of Blasco and Signes~\cite{bla-sig}
and Botelho and Santos \cite{bot-san} fixed the framework and solved a great part of the natural problems appearing in this context. In 
the particular case when $U$ is the injective tensor product $X \hat{\otimes}_\epsilon Y$
(naturally identified as a subspace of~$\mathcal{L}(X^*,Y)$), 
$(\ell^s_p,\ell_p)$-summing operators had been studied earlier by Kislyakov~\cite{kis} 
as ``$(p,Y)$-summing'' operators. In particular, he gave a Pietsch-type domination theorem
for $(\ell^s_p,\ell_p)$-summing operators defined on $X \hat{\otimes}_\epsilon Y$
(see \cite[Theorem~1.1.6]{kis}). This led to the natural question of whether
a Pietsch-type domination theorem holds for arbitrary $(\ell^s_p,\ell_p)$-summing operators, see \cite[Question~5.2]{bla-sig}.
Botelho and Santos extended Kislyakov's result by showing that this is the case when $U$ is Schwartz's $\epsilon$-product $X\epsilon Y$, i.e.
the subspace of all operators from~$X^*$ to~$Y$ which are ($w^*$-to-norm) continuous when restricted to~$B_{X^*}$
(see \cite[Theorem~3.1]{bot-san}). 

This paper is organized as follows.

In Section~\ref{section:Pietsch} we give new insight on the Botelho-Santos theorem
and we provide a negative answer to the aforementioned question, see Example~\ref{example:counterBS}.
To this end, we characterize those $(\ell^s_p,\ell_p)$-summing operators admitting
a Pietsch-type domination by means of the strong operator topology (Theorem~\ref{theorem:equiv}).
All of this is naturally connected with a discussion on measurability properties of operators which might be
of independent interest. 

In Section~\ref{section:Kwapien} we start a general analysis of the summability properties of operators defined on spaces of operators 
that imply similar properties for the adjoint maps. Our main result along this way is a Kwapie\'{n}-type theorem 
involving the special summation that arises in this setting related to the strong operator topology, see Theorem~\ref{theorem:equiv2}.

\subsubsection*{Notation and terminology}

All our Banach spaces are real and all our topological spaces are Hausdorff. By a {\em subspace} of a Banach space we mean a norm-closed linear subspace.
By an {\em operator} we mean a continuous linear map between Banach spaces.
The norm of a Banach space~$X$ is denoted by $\|\cdot\|_X$ or simply $\|\cdot\|$. We write
$B_X=\{x\in X:\|x\|\leq 1\}$ (the closed unit ball of~$X$). The topological dual of~$X$ is denoted by~$X^*$ and
we write $w^*$ for its weak$^*$-topology. The evaluation of a functional $x^*\in X^*$
at $x\in X$ is denoted by either $\langle x,x^*\rangle$ or $\langle x^*,x\rangle$.
We write $X\not \supseteq \ell_1$ to say that $X$ does not contain subspaces isomorphic to~$\ell_1$.
We denote by $\mathcal{L}(X^*,Y)$ the Banach space of all operators from~$X^*$ to~$Y$, equipped
with the operator norm. The {\em strong operator topology} ({\em SOT} for short) on $\mathcal{L}(X^*,Y)$
is the locally convex topology for which the sets 
$$
	\{T\in \mathcal{L}(X^*,Y): \, \|T(x^*)\|_Y<\epsilon\}, 
	\quad x^*\in X^*, 
	\quad \epsilon>0,
$$ 
are a subbasis of open neighborhoods of~$0$. That is, a net $(T_\alpha)$ in $\mathcal{L}(X^*,Y)$ is SOT-convergent to~$0$
if and only if $\|T_\alpha(x^*)\|_Y\to 0$ for every $x^*\in X^*$. Given a compact topological space~$L$, we denote by
$C(L)$ the Banach space of all real-valued continuous functions on~$L$, equipped with the supremum norm. 
Thanks to Riesz's representation theorem, the elements of $C(L)^*$ are identified with regular Borel signed measures on~$L$. 
We denote by $P(L) \sub C(L)^*$ the convex $w^*$-compact set of all regular Borel probability measures on~$L$.
For each $t\in L$, we write $\delta_t\in P(L)$ to denote the evaluation functional at~$t$, i.e. $\delta_t(h):=h(t)$
for all $h\in C(L)$. A function defined on~$L$ with values in a Banach space is said to be {\em universally strongly measurable} if it is
strongly $\mu$-measurable for all $\mu \in P(L)$. We will mostly consider the case when $L$ is the dual closed unit ball $B_{X^*}$
equipped with the weak$^*$-topology.

\section{Pietsch-type domination of $(\ell^s_p,\ell_p)$-summing operators}\label{section:Pietsch}

Throughout this section we fix $1\leq p<\infty$.

The aforementioned Pietsch-type domination theorem for $(\ell^s_p,\ell_p)$-summing operators proved in~\cite[Theorem~3.1]{bot-san} reads as follows:

\begin{theorem}[Botelho-Santos]\label{theorem:BS}
Let $U$ be a subspace of $X\epsilon Y$ and let $S:U\to Z$ be an $(\ell^s_p,\ell_p)$-summing operator. Then
there exist a constant $K\geq 0$ and $\mu \in P(B_{X^*})$ such that
\begin{equation}\label{eqn:BotelhoSantos}
	\|S(T)\|_Z \leq K \Big(\int_{B_{X^*}}\|T(\cdot)\|_{Y}^p \, d\mu\Big)^{1/p}
\end{equation}
for every $T\in U$.
\end{theorem}

A first comment is that the integral of inequality~\eqref{eqn:BotelhoSantos} is always well-defined for any $T\in X\epsilon Y$ and $\mu\in P(B_{X^*})$. Indeed, 
the restriction $T|_{B_{X^*}}$ is ($w^*$-to-norm) continuous, so
it is universally strongly measurable. Since in addition $T|_{B_{X^*}}$ is bounded, it belongs to the Lebesgue-Bochner space $L_p(\mu,Y)$. 

\begin{remark}\label{remark:BS}
Actually, Theorem~\ref{theorem:BS} is proved in~\cite[Theorem~3.1]{bot-san} for operators~$S$ defined on 
a subspace~$U$ contained in
$$
	\mathcal{L}_{w^*,\|\cdot\|}(X^*,Y)=\{T\in \mathcal L(X^*,Y):\, T \text{ is ($w^*$-to-norm) continuous}\}.
$$ 
The proof given there is based on the abstract Pietsch-type domination theorem of Botelho, Pellegrino and Rueda~\cite{bot-pel-rue},
and the argument works for subspaces of $X\epsilon Y$ as well. We stress that $\mathcal{L}_{w^*,\|\cdot\|}(X^*,Y)$
consists of finite rank operators, one has
$$
	\overline{\mathcal{L}_{w^*,\|\cdot\|}(X^*,Y)}^{\|\cdot\|}=X\hat{\otimes}_\epsilon Y \sub X \epsilon Y 
$$
and, in general, $\mathcal{L}_{w^*,\|\cdot\|}(X^*,Y)\neq X\epsilon Y$.
\end{remark}

We next provide a more direct proof of Theorem~\ref{theorem:BS}. Yet another approach will be presented at the end of this section.

\begin{proof}[Proof of Theorem~\ref{theorem:BS}] For any $n\in \N$ and $\bar{T}=(T_1,\dots,T_n)\in U^n$, we define
$$
	\Delta_{\bar{T}}: P(B_{X^*}) \to \erre, \quad
	\Delta_{\bar{T}}(\mu):=\sum_{i=1}^n \|S(T_i)\|_Z^p
	- K^{p}
	\int_{B_{X^*}} \sum_{i=1}^n \|T_i(\cdot)\|_{Y}^p \, d\mu,
$$
where $K\geq 0$ is a constant as in Definition~\ref{definition:pPettisSumming}.
Clearly, $\Delta_{\bar{T}}$ is convex and $w^*$-continuous, because the real-valued function 
$$
	x^*\mapsto \sum_{i=1}^n \|T_i(x^*)\|_Y^p
$$ 
is $w^*$-continuous on~$B_{X^*}$. This function attains its supremum at some $x_{\bar{T}}^*\in B_{X^*}$. Bearing
in mind that $S$ is $(\ell^s_p,\ell_p)$-summing, we get $\Delta_{\bar{T}}(\delta_{x^*_{\bar{T}}})\leq 0$. 

Note also that the collection of all functions
of the form $\Delta_{\bar{T}}$ is a convex cone in~$\mathbb{R}^{P(B_{X^*})}$. Indeed, given $\bar{T}=(T_1,\dots,T_n)\in U^n$,
$\bar{R}=(R_1,\dots,R_m)\in U^m$, $\alpha\geq 0$ and $\beta \geq 0$, we have
$\alpha\Delta_{\bar{T}}+\beta\Delta_{\bar{R}}=\Delta_{\bar{H}}$, where 
$$
	\bar{H}=(\alpha^{1/p}T_1,\dots,\alpha^{1/p}T_n,\beta^{1/p}R_1,\dots,\beta^{1/p}R_m).
$$
Therefore, by Ky Fan's Lemma (see e.g. \cite[Lemma~9.10]{die-alt}),
there is $\mu \in P(B_{X^*})$ such that $\Delta_{\bar{T}}(\mu)\leq 0$ for all functions of the form $\Delta_{\bar{T}}$.
In particular, inequality~\eqref{eqn:BotelhoSantos} holds for all $T\in U$.
\end{proof}

Clearly, in order to extend the statement of Theorem~\ref{theorem:BS} to other subspaces $U$ of $\mathcal{L}(X^*,Y)$,
the real-valued map $\|T(\cdot )\|_Y$ needs to be $\mu$-measurable for every $T\in U$. This holds automatically if $U$ is a subspace of
$$
	\mathcal{UM}(X^*,Y):=\{T\in \mathcal{L}(X^*,Y): \, T|_{B_{X^*}} \mbox{ is universally strongly measurable}\}.
$$
Note that $\mathcal{UM}(X^*,Y)$ is a SOT-sequentially closed subspace of $\mathcal{L}(X^*,Y)$.

\begin{example}\label{example:UM1}
\begin{enumerate}
\item[(i)] We have $X\epsilon Y \sub \mathcal{UM}(X^*,Y)$ according to the comment preceding Remark~\ref{remark:BS}. 
\item[(ii)] More generally, {\em every ($w^*$-to-weak) continuous operator from~$X^*$ to~$Y$
belongs to $\mathcal{UM}(X^*,Y)$.} Indeed, just bear in mind that any weakly continuous function from a compact topological space to a Banach space is universally strongly measurable,
see \cite[Proposition~4]{ari-alt}. We stress that, by the Banach-Dieudonn\'{e} theorem, an operator $T:X^*\to Y$ is ($w^*$-to-weak) continuous
if and only if the restriction $T|_{B_{X^*}}$ is ($w^*$-to-weak) continuous. 
\item[(iii)] In particular, {\em if $X$ is reflexive, then $\mathcal{L}(X^*,Y)=\mathcal{UM}(X^*,Y)$}.
\end{enumerate}
\end{example}

\begin{example}\label{example:UM2}
{\em If $X \not \supseteq \ell_1$, then every $T\in \mathcal{L}(X^*,Y)$ with separable range
belongs to~$\mathcal{UM}(X^*,Y)$.} Indeed, 
a result of Haydon~\cite{hay-J} (cf. \cite[Theorem~6.9]{van})
states that $X^{**}=\mathcal{UM}(X^*,\mathbb{R})$ if and only if $X\not\supseteq \ell_1$. The conclusion now follows from 
Pettis' measurability theorem applied to $T|_{B_{X^*}}$ and each $\mu\in P(B_{X^*})$, 
see e.g. \cite[p.~42, Theorem~2]{die-uhl-J}. 
\end{example}

So, we will look for conditions ensuring that an $(\ell^s_p,\ell_p)$-summing operator
defined on a subspace of $\mathcal{UM}(X^*,Y)$ is $(\ell^s_p,\ell_p)$-controlled, according to the following:

\begin{definition}\label{definition:dominated}
Let $U$ be a subspace of $\mathcal{UM}(X^*,Y)$. An operator $S: U \to Z$ is said to be {\em $(\ell^s_p,\ell_p)$-controlled}
if there exist a constant $K\geq 0$ and $\mu \in P(B_{X^*})$ such that
\begin{equation}\label{eqn:domi}
	\|S(T)\|_Z \leq K \Big(\int_{B_{X^*}} \|T(\cdot)\|_{Y}^p\, d\mu\Big)^{1/p} 
\end{equation}
for every $T\in U$.
\end{definition}

\begin{proposition}\label{proposition:facto}
Let $U$ be a subspace of $\mathcal{UM}(X^*,Y)$ and let $S: U \to Z$ be an operator.
Then $S$ is $(\ell^s_p,\ell_p)$-controlled if and only if there exist $\mu\in P(B_{X^*})$, a subspace $W \sub L_p(\mu,Y)$ and an operator $\tilde{S}:W \to Y$ such that $S$ factors as
$$
	
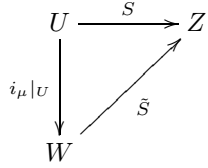
\xymatrix@R=3pc@C=3pc{U
	\ar[r]^{S} \ar[d]_{i_\mu|_U} & Z\\
	W  \ar@{->}[ur]_{\tilde{S}}  & \\
	}
$$
where $i_\mu:\mathcal{UM}(X^*,Y)\to L_p(\mu,Y)$ is the operator that maps each $T\in \mathcal{UM}(X^*,Y)$ to the equivalence
class of $T|_{B_{X^*}}$ in~$L_p(\mu,Y)$. 
\end{proposition}
\begin{proof}
It is clear that such factorization implies that $S$ is $(\ell^s_p,\ell_p)$-controlled. Conversely, inequality~\eqref{eqn:domi} in Definition~\ref{definition:dominated}
allows us to define a linear continuous map $\tilde{S}_0: i_\mu(U) \to Z$ by declaring $\tilde{S}_0(i_\mu(T)):=S(T)$ for all $T\in U$.
Now, we can extend $\tilde{S}_0$ to an operator $\tilde{S}$ from $W:=\overline{i_\mu(U)}$ to~$Z$. Clearly, we have $\tilde{S}\circ i_\mu|_U=S$.
\end{proof}

We next give a couple of applications of Proposition~\ref{proposition:facto} related to topological properties
of $(\ell^s_p,\ell_p)$-controlled operators.

The class of Banach spaces~$X$ such that $L_1(\mu)$ is separable for every $\mu \in P(B_{X^*})$ 
is rather wide. It contains, for instance, all weakly compactly generated Banach spaces
(cf. \cite[Theorem~13.20 and Corollary~14.6]{fab-ultimo}) 
as well as all Banach spaces not containing subspaces isomorphic to~$\ell_1$ (see \cite[Proposition~B.1]{avi-mar-ple}).
For such spaces we have:

\begin{corollary}\label{corollary:SeparableRange}
Suppose that $L_1(\mu)$ is separable for every $\mu \in P(B_{X^*})$ and that $Y$ is separable.
Let $U$ be a subspace of $\mathcal{UM}(X^*,Y)$ and let $S:U \to Z$ be an $(\ell^s_p,\ell_p)$-controlled operator. Then
$S$ has separable range.
\end{corollary}
\begin{proof}
Under such assumptions, $L_p(\mu,Y)$ is separable for any $\mu\in P(B_{X^*})$. The result
now follows from Proposition~\ref{proposition:facto}.
\end{proof}

A subset of a Banach space is said to be {\em weakly precompact} if every sequence in it 
admits a weakly Cauchy subsequence. Rosenthal's $\ell_1$-theorem~\cite{ros} (cf. \cite[Theorem~5.37]{fab-ultimo}) 
characterizes weakly precompact sets as those which are bounded and contain no sequence equivalent to the unit basis of~$\ell_1$.
An operator between Banach spaces is said to be {\em weakly precompact} if it maps bounded sets to
weakly precompact sets; this is equivalent to saying that it factors through a Banach space not containing subspaces isomorphic to~$\ell_1$.
For more information on weakly precompact operators we refer the reader to~\cite{gon-abe}.

\begin{corollary}\label{corollary:IdealProperties}
Let $U$ be a subspace of $\mathcal{UM}(X^*,Y)$ and let $S:U \to Z$ be an $(\ell^s_p,\ell_p)$-controlled operator. Then:
\begin{enumerate}
\item[(i)] $S$ is weakly compact whenever $Y$ is reflexive.
\item[(iii)] $S$ is weakly precompact whenever $Y \not\supseteq \ell_1$.
\end{enumerate}
\end{corollary}
\begin{proof} We consider a factorization of~$S$ as in Proposition~\ref{proposition:facto} and we distinguish two cases:

{\em Case $1<p<\infty$.} If $Y$ is reflexive, then so is $L_p(\mu,Y)$ (see e.g. \cite[p.~100, Corollary~2]{die-uhl-J}) and the same holds for~$W$, hence $S$ is weakly compact.
On the other hand, if $Y \not\supseteq \ell_1$, then $L_p(\mu,Y) \not\supseteq\ell_1$ (see e.g. \cite[Theorem~2.2.2]{cem-men})
and so $W\not\supseteq\ell_1$, hence $S$ is weakly precompact.

{\em Case $p=1$.} Let $j: L_2(\mu,Y)\to L_1(\mu,Y)$ be the identity operator. 
Since 
$$
	i_\mu(B_U) \sub j(B_{L_2(\mu,Y)}), 
$$
we deduce that $i_\mu(B_U)$ is relatively weakly compact
(resp. weakly precompact) whenever $Y$ is reflexive (resp. $Y \not\supseteq \ell_1$), and the same holds
for $S(B_U)=\tilde{S}(i_\mu(B_U))$.
\end{proof}

The following result shows the link between $(\ell^s_p,\ell_p)$-controlled
and $(\ell^s_p,\ell_p)$-summing operators.

\begin{theorem}\label{theorem:equiv}
Let $U$ be a subspace of~$\mathcal{UM}(X^*,Y)$ and let $S:U\to Z$ be an operator. Let us consider the following statements:
\begin{enumerate}
\item[(i)] $S$ is $(\ell^s_p,\ell_p)$-controlled. 
\item[(ii)] $S$ is $(\ell^s_p,\ell_p)$-summing and (SOT-to-norm) sequentially continuous.
\end{enumerate}
Then (i)$\impli$(ii). Moreover, both statements are equivalent whenever $U \cap X\epsilon Y$ is SOT-sequentially dense in~$U$.
\end{theorem}
\begin{proof}
Suppose first that $S$ is $(\ell^s_p,\ell_p)$-controlled and consider a factorization of~$S$ as in Proposition~\ref{proposition:facto}. 
We will deduce that $S$ is $(\ell^s_p,\ell_p)$-summing and (SOT-to-norm) sequentially continuous by checking that so is~$i_\mu$.
On one hand, $i_\mu$ is $(\ell^s_p,\ell_p)$-summing, because 
for every $n\in \N$ and $T_1,\dots,T_n\in \mathcal{UM}(X^*,Y)$ we have
$$
	\sum_{i=1}^n \|i_\mu(T_i)\|_{L_p(\mu,Y)}^p=
	\int_{B_{X^*}}\sum_{i=1}^n \|T_i(\cdot)\|_Y^p \, d\mu
	\leq \sup_{x^*\in B_{X^*}} \sum_{i=1}^n \|T_i(x^*)\|_Y^p.
$$
On the other hand, $i_\mu$ is (SOT-to-norm) sequentially continuous. Indeed, let $(T_n)$ be a sequence in~$\mathcal{UM}(X^*,Y)$
which SOT-converges to~$0$, i.e. $\|T_n(x^*)\|_Y\to 0$ for every $x^*\in X^*$. 
By the Banach-Steinhaus theorem, 
$\sup\{\|T_n\|:\, n\in\mathbb{N}\}<\infty$.
From Lebesgue's dominated convergence theorem it follows that $(i_\mu(T_n))$ converges to~$0$ in the norm topology of~$L_p(\mu,Y)$.

Suppose now that (ii) holds and that $U \cap X\epsilon Y$ is SOT-sequentially dense in~$U$.
The restriction $S|_{U \cap X\epsilon Y}$ is $(\ell^s_p,\ell_p)$-summing
and so Theorem~\ref{theorem:BS} and Proposition~\ref{proposition:facto} ensure the existence of $\mu \in P(B_{X^*})$, a 
subspace $W \sub L_p(\mu,Y)$ and an operator $\tilde{S}:W\to Z$ such that $i_\mu(U\cap X\epsilon Y)\sub W$ and 
$$
	\tilde{S}\circ i_\mu|_{U\cap X\epsilon Y}=S|_{U \cap X\epsilon Y}.
$$
Then we have $i_\mu(U)\sub W$ and $\tilde{S}\circ i_\mu|_{U}=S$, because
$S$ and $i_\mu$ are (SOT-to-norm) sequentially continuous
and $U \cap X\epsilon Y$ is SOT-sequentially dense in~$U$.
Therefore, $S$ is $(\ell^s_p,\ell_p)$-controlled.
\end{proof}

We are now ready to present a negative answer to \cite[Question 5.2]{bla-sig}:

\begin{example}\label{example:counterBS}
Suppose that $X$ is not reflexive and $X^*$ is separable (e.g. $X=c_0$).
Then $X^{**}=\mathcal{UM}(X^*,\mathbb{R})$, every $S\in X^{***}$ is $(\ell^s_p,\ell_p)$-summing, but
no $S\in X^{***}\setminus X^*$ is $(\ell^s_p,\ell_p)$-controlled (as operators from $X^{**}$ to~$\mathbb{R}$).
\end{example}
\begin{proof} The equality $X^{**}=\mathcal{UM}(X^*,\mathbb{R})$
follows from the fact that $X\not\supseteq\ell_1$, according to 
Haydon's result which we already mentioned in Example~\ref{example:UM2}.
Every $S\in X^{***}$ is easily seen to be $(\ell^s_p,\ell_p)$-summing as an operator from $X^{**}$ to~$\mathbb{R}$
(use that $B_{X^*}$ is $w^*$-dense in~$B_{X^{***}}$, by Goldstine's theorem).
On the other hand, if $S\in X^{***}$ is $(\ell^s_p,\ell_p)$-controlled, then it is $w^*$-sequentially continuous
by Theorem~\ref{theorem:equiv} (bear in mind that SOT$=w^*$ on~$X^{**}$).  
Since $(B_{X^{**}},w^*)$ is metrizable (because $X^*$ is separable), 
the restriction $S|_{B_{X^{**}}}$ is $w^*$-continuous and so, by the Banach-Dieudonn\'{e} theorem,
$S$ is $w^*$-continuous, i.e. $S\in X^*$.
\end{proof}

In order to apply Theorem~\ref{theorem:equiv}, there are many examples of subspaces $U$ of $\mathcal{UM}(X^*,Y)$
for which $U \cap X\epsilon Y$ is SOT-sequentially dense in~$U$. 
An operator $T:X^*\to Y$ is said to be {\em affine Baire-1} (we write $T\in \mathcal{AB}(X^*,Y)$ for short)
if there is a sequence in $X\epsilon Y$ which SOT-converges to~$T$. 
Affine Baire-1 operators were studied by Mercourakis and Stamati~\cite{mer-sta} and 
Kalenda and Spurn\'{y}~\cite{kal-spu}. We present below some examples. Recall first that a Banach space
$Y$ has the {\em approximation property} ({\em AP}) if for each norm-compact set $C \sub Y$ and each $\epsilon>0$
there is a finite rank operator $R:Y \to Y$ such that $\|R(y)-y\|_Y\leq \epsilon$ for all $y\in C$. If in addition $R$ can be chosen in such a way that
$\|R\| \leq \lambda$ for some constant $\lambda\geq 1$ (independent of~$C$ and~$\epsilon$), then $Y$ is said to have
the {\em $\lambda$-bounded approximation property} ({\em $\lambda$-BAP}). A Banach space is said to have the {\em bounded
approximation property} ({\em BAP}) if it has the $\lambda$-BAP for some $\lambda\geq 1$. For instance,
every Banach space with a Schauder basis has the BAP. In general, the AP and the BAP are different. However,
a separable dual Banach space has the AP if and only if it has the $1$-BAP. For more information
on these properties we refer the reader to~\cite{casazza}.

\begin{example}\label{example:weak}
{\em Suppose that $Y$ has the BAP. If $T \in \mathcal{L}(X^*,Y)$ is ($w^*$-to-weak) continuous and has separable range, then  
$T\in \mathcal{AB}(X^*,Y)$.}      
\end{example}
\begin{proof}
Let $\lambda\geq 1$ be a constant such that $Y$ has the $\lambda$-BAP.
Given any countable set $D \sub Y$, there is a sequence $(R_n)$ of finite rank operators on~$Y$ such that 
$\|R_n\|\leq \lambda$ for all $n\in \N$ and $\|R_n(y)-y\|_Y \to 0$ for every $y\in D$. Therefore,
$\|R_n(y)-y\|_Y \to 0$ for every $y\in \overline{D}$ (the norm-closure of~$D$). In particular, if this argument is applied to 
any countable set $D$ such that $D\sub T(X^*) \sub \overline{D}$, we get that the sequence $(R_n\circ T)$ is SOT-convergent to~$T$
in~$\mathcal{L}(X^*,Y)$. Note that each $R_n\circ T$ is ($w^*$-to-weak) continuous (because so is~$T$) and has finite rank, hence it 
belongs to~$\mathcal{L}_{w^*,\|\cdot\|}(X^*,Y)\sub X\epsilon Y$.
\end{proof}

\begin{example}\label{example:MS}
{\em Suppose that $X^*$ is separable and that either $X^*$ or $Y$ has the BAP. Then}
$$
	\mathcal{L}(X^*,Y) = \mathcal{AB}(X^*,Y),
$$ 
see \cite[Theorems~2.18 and~2.19]{mer-sta}. The proofs of these results contain a gap which
was commented and corrected in \cite[Remark~4.4]{kal-spu}. Note that the separability assumption on~$Y$ that appears 
in the statement of \cite[Theorem~2.19]{mer-sta} can be removed by using the arguments of~\cite{kal-spu}.
\end{example}

Clearly, $\mathcal{AB}(X^*,Y)$ is a linear subspace of $\mathcal{L}(X^*,Y)$. It is norm-closed whenever $Y$ has the BAP, as we next show.
To this end, we use an argument similar to the usual proof that the uniform limit of 
a sequence of real-valued Baire-1 functions is Baire-1 (see e.g. \cite[Proposition~A.126]{luk-alt}). However,
some technicalities arise since we need to approximate with operators instead of arbitrary continuous maps.   

\begin{lemma}\label{lem:closed}
\it If $Y$ has the BAP, then $\mathcal{AB}(X^*,Y)$ is norm-closed in $\mathcal{L}(X^*,Y)$.
\end{lemma}
\begin{proof}  
Fix $\lambda\geq 1$ such that $Y$ has the $\lambda$-BAP.
Let $T \in \overline{\mathcal{AB}(X^*,Y)}^{\|\cdot\|}$ with $\|T\|=1$. Let $(U_k)$ be a sequence 
in $\mathcal{AB}(X^*,Y)$ such that $\|U_k\|\leq 2^{-k+1}$ for all $k\in \Nat$ and $T=\sum_{k\in \N}U_k$ in the operator norm. 
Given $k\in \N$, we can apply to~$U_k$ the vector-valued version of Mokobodzki's theorem proved in \cite[Theorem~2.2]{kal-spu} to obtain 
a sequence $(S_{k,n})_{n\in \N}$ in $X\epsilon Y$ such that 
\begin{itemize}
\item $(S_{k,n})_{n\in \N}$ SOT-converges to~$U_k$;
\item $\|S_{k,n}\|\leq \lambda 2^{-k+1}$ for all $n\in \N$. 
\end{itemize}
Define a sequence $(T_n)$ in $X \epsilon Y$ by
$$
	T_n:=\sum_{k=1}^n S_{k,n}
	\quad\mbox{for all }n\in \N.
$$
It is easy to check that $(T_n)$ SOT-converges to~$T$, hence $T\in \mathcal{AB}(X^*,Y)$.
\end{proof}

As usual, we denote by $\mathcal{K}(X^*,Y)$ the subspace of $\mathcal{L}(X^*,Y)$ consisting of all compact operators from~$X^*$ to~$Y$.
Clearly, we have $X\epsilon Y \sub \mathcal{K}(X^*,Y)$.

\begin{example}\label{example:MScompact}
\it Suppose that $X$ is separable and $X\not \supseteq \ell_1$.
\begin{enumerate}
\item[(i)]  Every finite rank operator $T:X^* \to Y$ is affine Baire-1.
\item[(ii)] If $Y$ has the BAP, then 
$$
	\mathcal{K}(X^*,Y) \sub \mathcal{AB}(X^*,Y).
$$
\end{enumerate}
\end{example}
\begin{proof} (i) It suffices to check it for rank one operators. Fix $x^{**}\in X^{**}$ and $y\in Y$
in such a way that $T(x^*)=\langle x^{**},x^*\rangle y$ for all $x^*\in X^*$.
Since $X$ is $w^*$-sequentially dense in~$X^{**}$ 
(by the Odell-Rosenthal theorem~\cite{ode-ros}, cf. \cite[Theorem~4.1]{van}), there is
a sequence $(x_n)$ in~$X$ which $w^*$-converges to~$x^{**}$. For each $n\in \N$
we define $T_n\in \mathcal{L}_{w^*,\|\cdot\|}(X^*,Y) \sub X\epsilon Y$
by declaring $T_n(x^*):=\langle x_n,x^*\rangle y$ for all $x^*\in X^*$.
Clearly, $(T_n)$ is SOT-convergent to~$T$.

(ii) Take any $T\in \mathcal{K}(X^*,Y)$. Since $Y$ has the AP, 
there is a sequence $(T_n)$ of finite rank operators from~$X^*$ to~$Y$ 
converging to~$T$ in the operator norm. Each $T_n$ is affine Baire-1 by~(i). An appeal to Lemma~\ref{lem:closed}
ensures that $T\in \mathcal{AB}(X^*,Y)$.
\end{proof}

The proof of Theorem~\ref{theorem:BS} makes essential use of the $w^*$-continuity on~$B_{X^*}$
of the real-valued map $\|T(\cdot)\|_Y$ for $T\in X\epsilon Y$. We next present an abstract Pietsch-type domination theorem
for $(\ell^s_p,\ell_p)$-summing operators that does not require that continuity assumption, at the price
of dominating with a {\em finitely additive} measure. As a consequence of this result, we will obtain another proof of Theorem~\ref{theorem:BS}.

Given a measurable space~$(\Omega,\Sigma)$, we denote by
$B(\Sigma)$ the Banach space of all bounded $\Sigma$-measurable real-valued functions on~$\Omega$, equipped with the supremum norm.
The dual $B(\Sigma)^*$ can be identified with the Banach space ${\rm ba}(\Sigma)$ of all 
bounded finitely additive real-valued measures on~$\Sigma$, equipped with the variation norm. The duality is given
by integration, that is, $\langle h,\nu \rangle=\int_\Omega h \, d\nu$ for every $h\in B(\Sigma)$ and $\nu\in {\rm ba}(\Sigma)$, see e.g.
\cite[p.~77, Theorem~7]{die-J}.

\begin{theorem}\label{theorem:FA}
Let $\Sigma$ be a $\sigma$-algebra on~$B_{X^*}$ and let $U$ be a subspace of $\mathcal{L}(X^*,Y)$ 
such that the restriction of $\|T(\cdot)\|_Y$ to~$B_{X^*}$ is $\Sigma$-measurable for every $T\in U$.
Let $S:U \to Z$ be an $(\ell^s_p,\ell_p)$-summing operator.
Then there exist a constant $K\geq 0$ and a finitely additive probability $\nu$ on~$\Sigma$ such that
\begin{equation}\label{eqn:FA}
	\|S(T)\|_Z \leq K \Big(\int_{B_{X^*}} \|T(\cdot)\|_{Y}^p \, d\nu \Big)^{1/p} 
\end{equation}
for every $T\in U$.
\end{theorem}
\begin{proof}
For each $T\in U$ we define $\psi_T\in B(\Sigma)$ by 
$$
	\psi_T(x^*):=\|T(x^*)\|^p_Y \quad
	\mbox{for all }x^*\in B_{X^*}.
$$
Let $L \sub {\rm ba}(\Sigma)=B(\Sigma)^*$ be the convex $w^*$-compact set of all finitely additive probabilities on~$\Sigma$.
For any $n\in \N$ and $\bar{T}=(T_1,\dots,T_n)\in U^n$, we define
$$
	\Delta_{\bar{T}}: L \to \erre, \quad
	\Delta_{\bar{T}}(\nu):=\sum_{i=1}^n \|S(T_i)\|_Z^p
	- K^{p}
	\int_{K} \sum_{i=1}^n \psi_{T_i} \, d\nu,
$$
where $K\geq 0$ is a constant as in Definition~\ref{definition:pPettisSumming}.
Clearly, $\Delta_{\bar{T}}$ is convex and $w^*$-continuous. Moreover, 
by the Hahn-Banach theorem there is $\eta_{\bar{T}} \in{\rm ba}(\Sigma)$ 
with $\|\eta_{\bar{T}}\|_{{\rm ba}(\Sigma)}=1$ such that 
$$
	\Big\langle \sum_{i=1}^n \psi_{T_i},\eta_{\bar{T}} \Big\rangle=
	\Big\|\sum_{i=1}^n \psi_{T_i}\Big\|_{B(\Sigma)}.
$$
Bearing in mind that $\sum_{i=1}^n \psi_{T_i}\geq 0$, it follows that the variation $|\eta_{\bar{T}}| \in L$
satisfies
$$
	\Big\langle \sum_{i=1}^n \psi_{T_i},|\eta_{\bar{T}}| \Big\rangle=\sup_{x^*\in B_{X^*}}\sum_{i=1}^n \psi_{T_i}(x^*).
$$
Therefore, inequality~\eqref{eqn:psumming} in Definition~\ref{definition:pPettisSumming} yields
$$
	\Delta_{\bar{T}}\big(|\eta_{\bar{T}}|\big) = \sum_{i=1}^n \|S(T_i)\|_Z^p 
	- K^p \Big\langle \sum_{i=1}^n \psi_{T_i},|\eta_{\bar{T}}| \Big\rangle \leq 0.
$$

The collection of all functions of the form $\Delta_{\bar{T}}$ is easily seen to be a convex cone in~$\mathbb{R}^{L}$.
By Ky Fan's Lemma (see e.g. \cite[Lemma~9.10]{die-alt}),
there is $\nu \in L$ such that $\Delta_{\bar{T}}(\nu)\leq 0$ for all functions of the form $\Delta_{\bar{T}}$.
In particular, \eqref{eqn:FA} holds for every $T\in U$.
\end{proof}

\begin{proof}[Another proof of Theorem~\ref{theorem:BS}]
Let $\Sigma:={\rm Borel}(B_{X^*},w^*)$. 
Let $K$ and $\nu$ be as in Theorem~\ref{theorem:FA}. Define $\varphi\in B(\Sigma)^*$ by $\langle h,\varphi\rangle:=\int_{B_{X^*}}h \, d\nu$
for all $h\in B(\Sigma)$. Let $\mu\in C(B_{X^*})^*$ be the restriction of $\varphi$ to~$C(B_{X^*})$ (as a subspace of $B(\Sigma)$).
Then $\mu\in P(B_{X^*})$ and~\eqref{eqn:FA} now reads as
$$
	\|S(T)\|_Z \leq K \Big(\int_{B_{X^*}} \|T(\cdot)\|_{Y}^p \, d\mu \Big)^{1/p}
$$
for every $T\in U \sub X\epsilon Y$.
\end{proof}

\section{Kwapie\'{n}-type theorem for $(\ell^s_p,\ell^s_q)$-dominated operators}\label{section:Kwapien}

Throughout this section we fix $1< p, q< \infty$ such that $1/p + 1/q \leq 1$. Let $1\leq r < \infty$ be defined by $1/p + 1/q =1/r$. 
An operator 
$S:X\to Y$ is said to be {\em $(p,q)$-dominated} if there is a constant $K\geq 0$ such that
$$
	\Big( \sum_{i=1}^n | \langle S(x_i),y^*_i \rangle|^r \Big)^{1/r}
	\\ \le K
	\sup_{x^* \in B_{X^*}} \Big(\sum_{i=1}^n |\langle x_i,x^*\rangle|^p\Big)^{1/p}  \cdot 
 	\sup_{y \in B_{Y}} \Big(\sum_{i=1}^n | \langle y,y^*_i \rangle|^q\Big)^{1/q}
$$
for every $n\in \Nat,$ every $x_1,\dots,x_n \in X$ and every $y^*_1, \dots, y^*_n \in Y^*$.
The classical result of Kwapie\'{n}~\cite{kwa} mentioned in the introduction says that an operator between Banach spaces
is $(p,q)$-dominated if and only if it can be written as $S_1\circ S_2$ for some operators $S_1$ and $S_2$ such that $S_2$ is absolutely $p$-summing
and $S_1^*$ is absolutely $q$-summing (cf. \cite[\S 19]{def-flo}). Our aim in this section is
to extend Kwapie\'{n}'s result to the framework of $(\ell^s_p,\ell_p)$-summing operators, see Theorem~\ref{theorem:equiv2} below.

From now on we assume that $Z^*$ is a subspace of~$\mathcal{UM}(E^*,F)$ for some fixed Banach spaces $E$ and $F$.
Accordingly, the adjoint of any operator taking values in~$Z$ is defined on a subspace of~$\mathcal{UM}(E^*,F)$
and we can discuss whether it is $(\ell^s_q,\ell_q)$-summing or $(\ell^s_q,\ell_q)$-controlled.

\begin{definition}\label{definition:pqdom}
Let $U$ be a subspace of $\mathcal L(X^*,Y)$.
An operator $S: U \to Z$ is said to be {\em $(\ell^s_p,\ell^s_q)$-dominated} if there is a constant $K\geq 0$ such that
\begin{multline}\label{eqn:pqdom}
	\Big( \sum_{i=1}^n | \langle S(T_i),z^*_i \rangle|^r \Big)^{1/r}
	\\ \le K
	\sup_{x^* \in B_{X^*}} \Big(\sum_{i=1}^n \|T_i(x^*)\|_Y^p\Big)^{1/p}  \cdot 
 	\sup_{e^* \in B_{E^*}} \Big(\sum_{i=1}^n \| z^*_i (e^*)\|_{F}^q\Big)^{1/q} 
\end{multline}
for every $n\in \Nat,$ every $T_1,\dots,T_n \in U$ and every $z^*_1, \dots, z^*_n \in Z^*$.
\end{definition}

\begin{theorem}\label{theorem:equiv2} 
Let $U$ be a subspace of~$\mathcal{UM}(X^*,Y)$ and let $S:U\to Z$ be an operator. Consider the following statements:
\begin{enumerate}

\item[(i)] $S$ is $(\ell^s_p,\ell^s_q)$-dominated. 

\item[(ii)] There exist a constant $K\geq 0$ and measures  $\mu \in P(B_{X^*})$ and $\eta \in P(B_{E^{*}})$ such that
\begin{equation}\label{eqn:intpqdom}
	| \langle S(T), z^*\rangle|  \leq K \Big(\int_{B_{X^*}}\|T(\cdot)\|^p_{Y} \, d\mu\Big)^{1/p} 
	\cdot 
	\Big(\int_{B_{E^*}}\| z^*(\cdot)\|^q_{F} \, d\eta \Big)^{1/q}
\end{equation}
for every $T \in U\cap X\epsilon Y$ and every $z^* \in Z^*\cap E \epsilon F$.

\item[(iii)] There exist a constant $K\geq 0$ and measures  $\mu \in P(B_{X^*})$ and $\eta \in P(B_{E^{*}})$ such that
\eqref{eqn:intpqdom} holds for every $T \in U$ and every $z^* \in Z^*$.

\item[(iv)] There exist a Banach space $W$, an
$(\ell^s_p,\ell_p)$-controlled operator $S_2:U\to W$ and an operator
$S_1:W\to Z$ with $(\ell^s_q,\ell_q)$-controlled adjoint such that $S$ factors as $S= S_1 \circ S_2$.

\item[(v)] There exist a Banach space $W$, an
$(\ell^s_p,\ell_p)$-summing operator $S_2:U\to W$ and an operator
$S_1:W\to Z$ with $(\ell^s_q,\ell_q)$-summing adjoint such that $S$ factors as~$S= S_1 \circ S_2$.

\end{enumerate}
Then (iii)$\impli$(iv)$\impli$(v)$\impli$(i)$\impli$(ii). All statements are equivalent if, in addition, we assume that:
\begin{enumerate}
\item[(a)] the identity map on~$Z^*$ is (SOT-to-$w^*$) sequentially continuous;
\item[(b)] $Z^* \cap E\epsilon F$ is SOT-sequentially dense in~$Z^*$;
\item[(c)] $U \cap X\epsilon Y$ is SOT-sequentially dense in~$U$;
\item[(d)] $S$ is (SOT-to-norm) sequentially continuous.
\end{enumerate}
\end{theorem}

For the sake of brevity it is convenient to introduce the following:

\begin{definition}\label{definition:admissible}
We say that the triple $(Z,E,F)$ is {\em admissible} 
if conditions~(a) and~(b) above hold.
\end{definition}

Before embarking on the proof of Theorem~\ref{theorem:equiv2} we present some examples of 
admissible triples. Recall that the {\em weak operator topology} ({\em WOT} for short) on $\mathcal{L}(E^*,F)$
is the locally convex topology for which the sets 
$$
	\{R\in \mathcal{L}(E^*,F): \, |\langle R(e^*),f^*\rangle|<\epsilon\}, 
	\quad e^*\in E^*, \quad f^*\in F^*, 
	\quad \epsilon>0,
$$ 
are a subbasis of open neighborhoods of~$0$. So, a net $(R_\alpha)$ in $\mathcal{L}(E^*,F)$ is WOT-convergent to~$0$
if and only if $(R_\alpha(e^*))$ is weakly null in~$F$ for every $e^*\in E^*$.

\begin{example}\label{example:wot-vs-weak}
{\em If $Z^* \sub E \epsilon F$, then $(Z,E,F)$ is admissible.} Indeed, (b) holds trivially, while (a) follows from the fact that
a sequence in $E\epsilon F$ is WOT-convergent to~$0$ if and only if it is weakly null in~$E\epsilon F \sub \mathcal{L}(E^*,F)$
 (see e.g. \cite[Theorem~1.3]{col-rue}).
\end{example}

\begin{example}\label{example:triple-l1} 
Suppose that $E \not\supseteq \ell_1$. Take $Z:=E^*$ and $F:=\mathbb{R}$. 
Then we have $Z^{*} = E^{**} = \mathcal{UM}(E^{*},F)$ (see Example~\ref{example:UM2}) and, of course, SOT $=w^*$ on~$Z^*$,
so that (a) holds. If in addition $E$ is separable, then (b) also holds, i.e. $E\epsilon F = E$
is $w^*$-sequentially dense in $E^{**}$, by the Odell-Rosenthal theorem~\cite{ode-ros} (cf. \cite[Theorem~4.1]{van}). 
\end{example}

\begin{example}\label{example:projective}
Suppose that $F:=X_0^*$ for a Banach space~$X_0$. Take $Z:=E^* \hat{\otimes}_\pi X_0$ (the projective tensor product of~$E^*$ and~$X_0$).
Then: 
\begin{enumerate}
\item[(i)] $Z^*=\mathcal{L}(E^*,F)$ in the natural way (see e.g. \cite[p.~230, Corollary~2]{die-uhl-J}).
\item[(ii)] The identity map on~$Z^*$ is (WOT-to-$w^*$) sequentially continuous.
\item[(iii)] If $E^*$ is separable and either $E^*$ or~$F$ has the BAP, then $Z^*=\mathcal{UM}(E^*,F)$ and 
$(Z,E,F)$ is admissible. 
\end{enumerate}
\end{example}
\begin{proof} (ii) Let $(\varphi_n)$ be a sequence in~$Z^*=\mathcal{L}(E^*,F)$ which WOT-converges to~$0$. Then it is bounded
(by the Banach-Steinhaus theorem) and 
$$
	\langle e^*\otimes x_0,\varphi_n \rangle=\langle x_0,\varphi_n(e^*)\rangle \to 0
	\quad\mbox{for all }e^*\in E^*\mbox{ and }x_0\in X_0,
$$
hence $(\varphi_n)$ is $w^*$-null. 

(iii) Under such assumptions $E\epsilon F$ is SOT-sequentially dense in
$\mathcal{L}(E^*,F)$ (see Example~\ref{example:MS}). In particular, we have $\mathcal{L}(E^*,F)=\mathcal{UM}(E^*,F)$. 
Bearing in mind~(ii) it follows that $(Z,E,F)$ is admissible. 
\end{proof}

\begin{proof}[Proof of Theorem~\ref{theorem:equiv2}]

(iii)$\impli$(iv) By assumption we have 
$$
	|\langle S(T),z^*\rangle| \leq K \|i_\mu(T)\|_{L_p(\mu,Y)} \|z^*\|_{Z^*}
	\quad\mbox{for every }T\in U \mbox{ and }z^*\in Z^*,
$$
hence 
$$
	\|S(T)\|_Z \leq K \|i_\mu(T)\|_{L_p(\mu,Y)}
	\quad \mbox{for every }T\in U.
$$ 
Write $W:=\overline{i_\mu(U)}$. By the previous inequality, there is an operator 
$S_1:  W \to Z$ such that $S_1\circ i_\mu|_U=S$
(cf. the proof of Proposition~\ref{proposition:facto}). 
Of course, $S_2:=i_\mu|_U$ is $(\ell^s_p,\ell_p)$-controlled.
We claim that $S_1^*:Z^* \to W^*$
is $(\ell^s_q,\ell_q)$-controlled. 
Indeed, inequality~\eqref{eqn:intpqdom} reads as
$$
	|\langle i_\mu(T), S_1^* (z^*) \rangle | \le K \, \|i_\mu(T)\|_{L_p(\mu,Y)} \, \| i_\eta(z^*)\|_{L_q(\eta,F)}
$$
for every $T\in U$ and $z^*\in Z^*$.
Thus, $\|S_1^*(z^*)\|_{W^*} \leq K \| i_\eta(z^*)\|_{L_q(\eta,F)}$
for every $z^*\in Z^*$, so that $S_1^*$
is $(\ell^s_q,\ell_q)$-controlled.

(iv)$\impli$(v) This follows from Theorem~\ref{theorem:equiv}.

(v)$\impli$(i) Fix $n\in \Nat$ and take $T_1,\dots,T_n \in U$ and $z^*_1, \dots, z^*_n \in Z^*$. Then Holder's inequality 
and the fact that $S_2$ (resp.~$S_1^*$) is $(\ell^s_p,\ell_p)$-summing
(resp. $(\ell^s_q,\ell_q)$-summing) yield
\begin{multline*}
	\Big(\sum_{i=1}^n | \langle S(T_i),z^*_i \rangle|^r\Big)^{1/r}=
	\Big(\sum_{i=1}^n | \langle S_2(T_i),S_1^*(z^*_i) \rangle|^r\Big)^{1/r}
	\\\leq \Big(\sum_{i=1}^n \|S_2(T_i)\|_W^r \cdot\|S_1^*(z^*_i)\|_{W^*}^r\Big)^{1/r}
	\leq \Big( \sum_{i=1}^n \|S_2(T_i)\|_W^p\Big)^{1/p}
	\cdot \Big( \sum_{i=1}^n \|S_1^*(z_i^*)\|_{W^*}^q\Big)^{1/q}
	\\ \leq K \sup_{x^* \in B_{X^*}} \Big(\sum_{i=1}^n \|T_i(x^*)\|_Y^p\Big)^{1/p}  \cdot 
 	\sup_{e^* \in B_{E^*}} \Big(\sum_{i=1}^n \| z^*_i (e^*)\|_{F}^q\Big)^{1/q}
\end{multline*}
for some constant $K\geq 0$ independent of the $T_i$'s and $z_i^*$'s. This shows that
$S$ is $(\ell^s_p,\ell^s_q)$-dominated.

(i)$\impli$(ii) Observe that $L:=P(B_{X^*}) \times P(B_ {E^{*}})$ is a compact convex set of the locally convex space
$C(B_{X^*})^* \times C(B_{E^*})^*$, equipped with the product of the corresponding $w^*$-topologies. Fix $n\in \N$, 
$$
	\bar{T}=(T_1,\dots,T_n ) \in (U\cap X\epsilon Y)^n
	\ \ \mbox{and} 
	\ \ \bar{z^*}=(z^*_1,\dots, z^*_n) \in (Z^*\cap E\epsilon F)^n.
$$ 
Consider the function $\Delta_{\bar{T}, \bar{z^*} }: L  \to \erre$ 
given by 
$$
	\Delta_{\bar{T},\bar{z^*} }(\mu,\eta):=  
$$
$$
\sum_{i=1}^n    |\langle S(T_i), z^*_i \rangle|^r
	- K^r \frac{r}{p} 
	\int_{B_{X^*}} \sum_{i=1}^n \|T_i(\cdot)\|_{Y}^p \, d\mu - K^r \frac{r}{q}  \int_{{B_{E^{*}}}} \sum_{i=1}^n \|z^*_i(\cdot)\|_{F}^q \, d\eta,
$$
where $K\geq 0$ is a constant as in Definition~\ref{definition:pqdom}.
Clearly, $\Delta_{\bar{T}, \bar{z^*} }$ is convex and continuous,
because $T_i\in X\epsilon Y$ and $z_i^*\in E\epsilon F$ for every~$i=1,\dots,n$. 
We claim that $\Delta_{\bar{T}, \bar{z^*}}(\mu,\eta)\leq 0$ for some $(\mu,\eta)\in  L$.
Indeed, since the functions 
$$
	x^*\mapsto \sum_{i=1}^n \|T_i(x^*)\|_Y^p \quad\mbox{and}\quad  e^{*} \mapsto \sum_{i=1}^n \|z^*_i(e^*)\|_F^q
$$ 
are $w^*$-continuous on~$B_{X^*}$ and $B_{E^{*}}$, they attain their suprema at some $x_{\bar{T}}^*\in B_{X^*}$
and $e^{*}_{\bar{z^*}} \in B_{E^{*}}$, respectively. By taking into account Young's inequality, we have 
\begin{multline}\label{eqn:ultima}
	\sum_{i=1}^n | \langle S(T_i),z^*_i \rangle|^r
	\stackrel{\eqref{eqn:pqdom}}{\le} K^r
	\Big(\sum_{i=1}^n \|T_i(x_{\bar{T}}^*)\|_Y^p\Big)^{r/p}  \cdot 
 	\Big(\sum_{i=1}^n \|z^*_i (e^{*}_{\bar{z^*}})\|_{F}^q\Big)^{r/q}
 	\\
 	\leq 
 	K^r\frac{r}{p}\sum_{i=1}^n \|T_i(x_{\bar{T}}^*)\|_Y^p+K^r\frac{r}{q}\sum_{i=1}^n \|z^*_i (e^{*}_{\bar{z^*}})\|_{F}^q.
\end{multline}
If we write $\mu:=\delta_{x^*_{\bar{T}}}\in P(B_{X^*})$
and $\eta:=\delta_{e^{*}_{\bar{z^*}}}\in P(B_{E^*})$, then \eqref{eqn:ultima} yields $\Delta_{\bar{T}, \bar{z^*}}(\mu,\eta)\leq 0$, as required.

The collection $\mathcal{C}$ of all functions
$\Delta_{\bar{T},\bar{z^*}}$ as above is a convex cone in $\mathbb{R}^{L}$. Indeed, $\mathcal{C}$ 
is obviously closed under sums and we have 
$$
	\alpha\Delta_{\bar{T},\bar{z^*}}=\Delta_{(\alpha^{1/p}T_1,\dots,\alpha^{1/p}T_n),(\alpha^{1/q}z_1^*,\dots,\alpha^{1/q}z_n^*)}
$$
for all $\alpha\geq 0$.

By Ky Fan's Lemma (see e.g. \cite[Lemma~9.10]{die-alt}), there is  
$(\mu, \eta)  \in L$ such that $\Delta_{\bar{T},\bar{z^*}}(\mu,\eta) \leq 0$ for 
every $\Delta_{\bar{T},\bar{z^*}}\in \mathcal{C}$. 
In particular,
\begin{multline}\label{eqn:fromKF}
 	|\langle S(T), z^*\rangle|^r \le 
 	K^r \frac{r}{p} 
	\int_{B_{X^*}}  \|T(\cdot)\|_{Y}^p \, d\mu + K^r \frac{r}{q}  \int_{B_{E^{*}}}  \|z^*(\cdot)\|_{F}^q \, d\eta
	\\ \quad\mbox{for every }T \in U\cap X \epsilon Y
	\mbox{ and }z^*\in Z^*\cap E \epsilon F.
\end{multline}

Fix $T \in U\cap X \epsilon Y$ and $z^*\in Z^*\cap E \epsilon F$. We will check that \eqref{eqn:intpqdom} holds. Write 
$$
	a:= \Big(\int_{B_{X^*}}  \|T(\cdot)\|_{Y}^p \, d\mu \Big)^{1/p}
	\quad\mbox{and}\quad
	b:= \Big(\int_{B_{E^{*}}}  \|z^*(\cdot)\|_{F}^q \, d\eta \Big)^{1/q}.
$$ 
If either $a=0$ or $b=0$, then $\langle S(T), z^*\rangle=0$. Indeed, if $a=0$, then 
for each $n\in \mathbb{N}$ inequality~\eqref{eqn:fromKF} applied to the pair $(nT,z^*)$ yields
$$
	|\langle S(T), z^*\rangle|^r =\frac{1}{n^r}\cdot |\langle S(nT), z^*\rangle|^r \leq  \frac{1}{n^r} \cdot \frac{K^rrb^q}{q},
$$
hence $\langle S(T), z^*\rangle=0$.
A similar argument works for the case $b=0$. On the other hand, if $a\neq 0$ and $b\neq 0$, then 
inequality~\eqref{eqn:fromKF} applied to the pair $(\frac{1}{a}T,\frac{1}{b}z^*)$ yields
$$
	|\langle S(T), z^*\rangle|^r  = a^r \, b^r \, \Big|\Big\langle S\Big(\frac{1}{a}T\Big),\frac{1}{b}z^*\Big\rangle\Big|^r 
$$
$$
	\le
	K^r \, a^r \, b^r  \left( \frac{r}{p \, a^p} 
		\int_{B_{X^*}}  \|T(\cdot)\|_{Y}^p \, d\mu + \frac{r}{q \, b^q}  \int_{B_{E^{*}}}  \|z^*(\cdot) \|_{F}^q \, d\eta \right)
	=
	K^r \, a^r \, b.^r
$$
This proves~\eqref{eqn:intpqdom} when $T \in U\cap X \epsilon Y$ and $z^*\in Z^*\cap E \epsilon F$.

Finally, we prove the implication (ii)$\impli$(iii) under the additional assumptions. Fix $T \in U$ and $z^*\in Z^*$. 
By~(c) (resp.~(b)), we can take a sequence $(T_n)$ (resp. $(z_n^*)$) in 
$U\cap X \epsilon Y$ (resp. $Z^*\cap E \epsilon F$)
which SOT-converges to~$T$ (resp.~$z^*$). For each $n\in \Nat$ we have
\begin{equation}\label{eqn:tothelimit}
	|\langle S(T_n),z_n^*\rangle|
	\leq 
	K \Big(\int_{B_{X^*}}\|T_n(\cdot)\|_Y^p \, d\mu\Big)^{1/p} \cdot \Big(\int_{B_{E^*}}\|z_n^*(\cdot)\|_F^q \, d\eta\Big)^{1/q}.
\end{equation}
Since the operators $i_\mu$ and $i_\eta$ are (SOT-to-norm) sequentially continuous (see the proof of Theorem~\ref{theorem:equiv}), we have
$$
	\lim_{n\to \infty} \Big(\int_{B_{X^*}}\|T_n(\cdot)\|_Y^p \, d\mu\Big)^{1/p} = \Big(\int_{B_{X^*}}\|T(\cdot)\|_Y^p \, d\mu\Big)^{1/p}
$$
and
$$ 
	\lim_{n\to \infty} \Big(\int_{B_{E^*}}\|z_n^*(\cdot)\|_F^q \, d\eta\Big)^{1/q}  = \Big(\int_{B_{E^*}}\|z^*(\cdot)\|_F^q \, d\eta\Big)^{1/q}.
$$
Moreover, $S$ is (SOT-to-norm) sequentially continuous by assumption~(d), so the sequence $(S(T_n))$ converges to~$S(T)$ in the norm topology.
Since $(z_n^*)$ is $w^*$-convergent to~$z^*$ (by~(a)), we conclude that
\begin{multline*}
	|\langle S(T),z^*\rangle| =\lim_{n\to \infty}
	|\langle S(T_n),z_n^*\rangle|
	\\ \stackrel{\eqref{eqn:tothelimit}}{\leq} 
	K \Big(\int_{B_{X^*}}\|T(\cdot)\|_Y^p \, d\mu\Big)^{1/p} \cdot \Big(\int_{B_{E^*}}\|z^*(\cdot)\|_F^q \, d\eta\Big)^{1/q},
\end{multline*}
as we wanted. The proof is finished. 
\end{proof}

\begin{remark}\label{remark:SOT-norm}
Statement (iv) in Theorem~\ref{theorem:equiv2} implies that $S_2$ is (SOT-to-norm) sequentially continuous (by Theorem~\ref{theorem:equiv}) and so is~$S$. 
\end{remark}

\begin{corollary}\label{corollary:Kwapien}
Suppose that $Z^* \sub E \epsilon F$. Let $U$ be a subspace of~$X\epsilon Y$ and let $S:U\to Z$ be an operator. 
Then the following statements are equivalent:
\begin{enumerate}

\item[(i)] $S$ is $(\ell^s_p,\ell^s_q)$-dominated. 

\item[(ii)] There exist a constant $K\geq 0$ and measures  $\mu \in P(B_{X^*})$ and $\eta \in P(B_{E^{*}})$ such that
$$
	| \langle S(T), z^*\rangle|  \leq K \Big(\int_{B_{X^*}}\|T(\cdot)\|^p_{Y} \, d\mu\Big)^{1/p} 
	\cdot 
	\Big(\int_{B_{E^*}}\| z^*(\cdot)\|^q_{F} \, d\eta \Big)^{1/q}
$$
for every $T \in U$ and every $z^* \in Z^*$.
\item[(iii)] There exist a Banach space $W$, an
$(\ell^s_p,\ell_p)$-summing operator $S_2:U\to W$ and an operator
$S_1:W\to Z$ with $(\ell^s_q,\ell_q)$-summing adjoint such that $S$ factors as~$S= S_1 \circ S_2$.
\end{enumerate}
\end{corollary}

\subsection*{Acknowledgements}
The authors thank J.M. Calabuig and P. Rueda for valuable discussions at the early stage of this work.
Research partially supported by {\em Agencia Estatal de Investigaci\'{o}n} [MTM2017-86182-P to J.R. 
and MTM2016-77054-C2-1-P to E.A.S.P., both grants cofunded by ERDF, EU]; and {\em Fundaci\'on S\'eneca} [20797/PI/18 to J.R.]


\begin{thebibliography}{10}

\bibitem{ari-alt}
J.~Arias~de Reyna, J.~Diestel, V.~Lomonosov, and L.~Rodr\'{\i}guez-Piazza,
  \emph{Some observations about the space of weakly continuous functions from a
  compact space into a {B}anach space}, Quaestiones Math. \textbf{15} (1992),
  no.~4, 415--425.

\bibitem{avi-mar-ple}
A.~Avil\'{e}s, G.~Mart\'{\i}nez-Cervantes, and G.~Plebanek, \emph{Weakly
  {R}adon-{N}ikod\'{y}m {B}oolean algebras and independent sequences}, Fund.
  Math. \textbf{241} (2018), no.~1, 45--66.

\bibitem{bla-sig}
O.~Blasco and T.~Signes, \emph{Some classes of {$p$}-summing type operators},
  Bol. Soc. Mat. Mexicana (3) \textbf{9} (2003), no.~1, 119--133.

\bibitem{bot-pel-rue}
G.~Botelho, D.~Pellegrino, and P.~Rueda, \emph{A unified {P}ietsch domination
  theorem}, J. Math. Anal. Appl. \textbf{365} (2010), no.~1, 269--276.

\bibitem{bot-san}
G.~Botelho and J.~Santos, \emph{A {P}ietsch domination theorem for
  {$(\ell_p^s,\ell_p)$}-summing operators}, Arch. Math. (Basel) \textbf{104}
  (2015), no.~1, 47--52.

\bibitem{casazza}
P.~G. Casazza, \emph{Approximation properties}, Handbook of the geometry of
  {B}anach spaces, {V}ol. {I}, North-Holland, Amsterdam, 2001, pp.~271--316.

\bibitem{cem-men}
P.~Cembranos and J.~Mendoza, \emph{Banach spaces of vector-valued functions},
  Lecture Notes in Mathematics, vol. 1676, Springer-Verlag, Berlin, 1997.

\bibitem{col-rue}
H.~S. Collins and W.~Ruess, \emph{Weak compactness in spaces of compact
  operators and of vector-valued functions}, Pacific J. Math. \textbf{106}
  (1983), no.~1, 45--71.

\bibitem{def-flo}
A.~Defant and K.~Floret, \emph{Tensor norms and operator ideals}, North-Holland
  Mathematics Studies, vol. 176, North-Holland Publishing Co., Amsterdam, 1993.

\bibitem{die-J}
J.~Diestel, \emph{Sequences and series in {B}anach spaces}, Graduate Texts in
  Mathematics, vol.~92, Springer-Verlag, New York, 1984.
  
\bibitem{die-alt}
J.~Diestel, H.~Jarchow, and A.~Tonge, \emph{Absolutely summing operators},
  Cambridge Studies in Advanced Mathematics, vol.~43, Cambridge University
  Press, Cambridge, 1995.

\bibitem{die-uhl-J}
J.~Diestel and J.~J. Uhl, Jr., \emph{Vector measures}, Mathematical Surveys, No. 15, American Mathematical
  Society, Providence, R.I., 1977.

\bibitem{fab-ultimo}
M.~Fabian, P.~Habala, P.~H{\'a}jek, V.~Montesinos, and V.~Zizler, \emph{Banach
  space theory}, CMS Books in Mathematics/Ouvrages de Math\'ematiques de la
  SMC, Springer, New York, 2011.

\bibitem{gon-abe}
M.~Gonz{\'a}lez and A.~Mart{\'{\i}}nez-Abej{\'o}n, \emph{Tauberian operators},
  Operator Theory: Advances and Applications, vol. 194, Birkh\"auser Verlag,
  Basel, 2010.
  
\bibitem{hay-J}
R.~Haydon, \emph{Some more characterizations of {B}anach spaces containing
  {$l\sb{1}$}}, Math. Proc. Cambridge Philos. Soc. \textbf{80} (1976), no.~2,
  269--276.

\bibitem{kal-spu}
O.~F.~K. Kalenda and J.~Spurn\'{y}, \emph{Baire classes of affine vector-valued
  functions}, Studia Math. \textbf{233} (2016), no.~3, 227--277.

\bibitem{kis}
S.~V. Kislyakov, \emph{Absolutely summing operators on the disc algebra},
  St. Petersburg Math. J. \textbf{3} (1992), no.~4, 705--774.

\bibitem{kwa}
S.~Kwapie\'{n}, \emph{On operators factorizable through {$L\sb{p}$} space},
  Bull. Soc. Math. France, M\'{e}m. No.~31--32 (1972), 215--225. 
  
\bibitem{luk-alt}
J.~Luke\v{s}, J.~Mal\'{y}, I.~Netuka, and J.~Spurn\'{y}, \emph{Integral
  representation theory}, De Gruyter Studies in Mathematics, vol.~35, Walter de
  Gruyter \& Co., Berlin, 2010.

\bibitem{mer-sta}
S.~Mercourakis and E.~Stamati, \emph{Compactness in the first {B}aire class and
  {B}aire-1 operators}, Serdica Math. J. \textbf{28} (2002), no.~1, 1--36.

\bibitem{ode-ros}
E.~Odell and H.~P. Rosenthal, \emph{A double-dual characterization of separable
  {B}anach spaces containing {$l\sp{1}$}}, Israel J. Math. \textbf{20} (1975),
  no.~3-4, 375--384.

\bibitem{ros}
H.~P. Rosenthal, \emph{A characterization of {B}anach spaces containing
  {$l\sp{1}$}}, Proc. Nat. Acad. Sci. U.S.A. \textbf{71} (1974), 2411--2413.

\bibitem{van}
D.~van Dulst, \emph{Characterizations of {B}anach spaces not containing {$l{\sp
  1}$}}, CWI Tract, vol.~59, Centrum voor Wiskunde en Informatica, Amsterdam,
  1989.

\end{thebibliography}

\bibliographystyle{amsplain}

\end{document}